\documentclass[12pt,reqno]{amsart}

\usepackage{enumitem} 
\setlist{leftmargin=24pt}

\usepackage{amsmath,amsfonts,amssymb,amscd,amsthm,calc}
\usepackage[english]{babel}
\usepackage{yfonts,color}

\usepackage{tikz}

\newtheorem{theorem}{Theorem}[section]
\newtheorem{proposition}[theorem]{Proposition}
\newtheorem{lemma}[theorem]{Lemma}
\newtheorem{corollary}[theorem]{Corollary}

\newtheorem{question}[theorem]{Question}

\theoremstyle{definition}
\newtheorem{definition}[theorem]{Definition}

\newcommand{\darrow}{\!\downarrow}
\newcommand{\uarrow}{\!\uparrow}
\newcommand{\la}{\langle}
\newcommand{\ra}{\rangle}

\newcommand{\bigset}[1]{\big\{ #1 \big\}}
\newcommand{\restr}{\mbox{\raisebox{.5mm}{$\upharpoonright$}}}

\renewcommand{\leq}{\leqslant}
\renewcommand{\geq}{\geqslant}

\newcommand{\fa}{\forall}
\newcommand{\ex}{\exists}
\newcommand{\vph}{\varphi}
\newcommand{\imp}{\rightarrow}

\newcommand{\A}{\mathcal{A}}
\newcommand{\B}{\mathcal{B}}
\newcommand{\E}{\mathcal{E}}
\newcommand{\G}{\mathcal{G}}
\newcommand{\K}{\mathcal{K}}

\newcommand{\PP}{\mathcal{P}}

\newcommand{\PA}{\mathrm{PA}}
\newcommand{\con}{\mathrm{con}}
\newcommand{\Cmpl}{\mathrm{Cmpl}}

\begin{document}

\title[Complexity of completions in pca]
{The complexity of completions in partial combinatory algebra}

\author[S. A. Terwijn]{Sebastiaan A. Terwijn}
\address[Sebastiaan A. Terwijn]{Radboud University Nijmegen\\
Department of Mathematics\\
P.O. Box 9010, 6500 GL Nijmegen, the Netherlands.
} \email{terwijn@math.ru.nl}

\begin{abstract} 
We discuss the complexity of completions of partial combinatory 
algebras, in particular of Kleene's first model. 
Various completions of this model exist in the literature, 
but all of them have high complexity. We show that although 
there do not exist computable completions, there exists 
completions of low Turing degree. 
We use this construction to relate completions of 
Kleene's first model to complete extensions of $\PA$. 
We also discuss the complexity of pcas defined from 
nonstandard models of $\PA$.
\end{abstract}



\date{\today}

\maketitle

\section{Introduction}

Partial combinatory algebra (pca) generalizes the setting of 
classical combinatory algebra to structures with a partial 
application operator. 
The first entry in the literature is Feferman~\cite{Feferman}, 
which is surprisingly late,
some fifty years after the invention of combinatory algebra and 
the closely related lambda calculus, although the concept of a 
pca existed before that (see section~\ref{sec:cmpl}).
Apart from this connection with lambda calculus, pcas have 
played a notable part in constructive mathematics. At the 
end of section~\ref{sec:pas} below, we list a number of key 
examples of pcas, and say something about their role in 
various settings. 

Since the application operator in pcas is partial, they can 
often be naturally represented as c.e.\ structures in the 
sense of Selivanov \cite{Selivanov} and Khoussainov~\cite{Khoussainov}. 
The computable structure theory of pcas as partial c.e.\ structures
was recently studied by Fokina and Terwijn~\cite{FokinaTerwijn}.
Since pcas can be seen as abstract models of computation, it is 
only natural to consider their complexity as algebraic structures 
from the viewpoint of computability theory. 
At least for countable pcas there is a straightforward definition 
of their complexity in terms of the complexity of a presentation, 
as in computable model theory. 
Below we formulate this using numberings (Definition~\ref{def:computable}).
Some of the complexity of pcas was studied earlier in 
Shafer and Terwijn \cite{ShaferTerwijn} and 
Golov and Terwijn \cite{GolovTerwijn}. 
In the current paper we focus on the complexity of completions, 
and in particular of completions of what is called Kleene's first 
model $\K_1$, with application defined in terms of partial computable 
functions on the natural numbers. 

A completion of a pca is a total pca (i.e.\ a combinatory algebra 
in the classical sense, in which applications are always defined) 
in which the pca can be embedded. 
Not every pca has a completion, as was first proved in Klop~\cite{Klop}. 
On the other hand, Kleene's $\K_1$ {\em does\/} have a completion. 
This follows from the sufficient condition for completability given 
in Bethke et {al.}~\cite{BethkeKlopVrijer1996}. 
This yields a certain term model $T(\omega)/\!\!\sim$ as a 
completion of~$\K_1$. 
In section~\ref{sec:cmpl} below we will discuss how Scott's 
graph model, which is another important example of a pca, 
can also be seen as a (weak) completion of $\K_1$.
We note, however, that these completions of $\K_1$ have 
high complexity, which brings up the question what the optimal 
complexity of such a completion could be. 
Although no computable completions of $\K_1$ exist 
(cf.\ Theorem~\ref{thm:noncomp} and the remarks following it), 
we show that there exist completions of $\K_1$ of low Turing degree 
(Theorem~\ref{low}). Such completions are close to computable in 
the sense that the complexity of their halting problem is the 
same as the standard halting problem. 

All this suggest a connection with complete extensions of 
Peano arithmetic, for which a similar story exists. 
Note however that we are talking here about pcas, i.e.\ the 
models of a theory, rather than the theory itself. 
Nevertheless, in section~\ref{sec:PA} we show that indeed there 
is a connection. 
Finally, in section~\ref{sec:nonstandard} we discuss the complexity 
of pcas resulting from nonstandard models of $\PA$.

Our notation from computability theory is mostly standard. 
For unexplained notions we refer to 
Odifreddi \cite{Odifreddi} or Soare \cite{Soare}.
In particular $\omega$ denotes the natural numbers, and
$\vph_e$ the $e$-th partial computable (p.c.) function.
For any set $A$, $A'$ denotes the Turing jump of $A$, 
and in particular $\emptyset'$ denotes the halting problem.

\section{Partial combinatory algebras} \label{sec:pas}

To make the paper self-contained, we briefly review the basic 
definitions from partial combinatory algebra. Our presentation 
follows van Oosten \cite{vanOosten}.

A {\em partial applicative structure\/} (pas) is a set $\A$ together
with a partial map $\cdot$ from $\A\times \A$ to $\A$.  
We usually write $ab$ instead of $a\cdot b$, and think of this as 
`$a$ applied to $b$'. If this is defined we denote this by $ab\darrow$.  
By convention, application associates to the left, 
so we write $abc$ instead of $(ab)c$.  
{\em Terms\/} over $\A$ are built from elements of $\A$,
variables, and application. If $t_1$ and $t_2$ are terms then so is
$t_1t_2$.  If $t(x_1,\ldots,x_n)$ is a term with variables $x_i$, and
$a_1,\ldots,a_n \in\A$, then $t(a_1,\ldots,a_n)$ is the term obtained
by substituting the $a_i$ for the~$x_i$.  For closed terms
(i.e.\ terms without variables) $t$ and $s$, we write $t \simeq s$ if
either both are undefined, or both are defined and equal. 
Here application is \emph{strict} in the sense that for $t_1t_2$ to be 
defined, it is required that both $t_1$ and $t_2$ are defined.

\begin{definition}\label{def}
A pas $\A$ is called {\em combinatory complete\/} if for any term
$t(x_1,\ldots,x_n,x)$, $n\geq 0$, with free variables among
$x_1,\ldots,x_n,x$, there exists $b\in \A$ such that
for all $a_1,\ldots,a_n,a\in \A$,
\begin{enumerate}[label={\rm (\roman*)}] 

\item $ba_1\cdots a_n\darrow$,

\item $ba_1\cdots a_n a \simeq t(a_1,\ldots,a_n,a)$.

\end{enumerate}
A pas $\A$ is a {\em partial combinatory algebra\/} (pca) if
it is combinatory complete. 
A combinatory algebra (ca) is a pca for which the application 
operator is total. 
\end{definition}

Combinatory completeness of pcas can be characterized by the 
existence of combinators $k$ and $s$, just as in classical 
combinatory algebra and lambda calculus. 
In \cite{vanOosten} it is stated that the following theorem is 
``essentially due to Feferman~\cite{Feferman}.'' 

\begin{theorem} {\rm (Feferman)} \label{Feferman}
A pas $\A$ is a pca if and only if it has elements $k$ and $s$
with the following properties for all $a,b,c\in\A$:
\begin{itemize}

\item $ka\darrow$ and $kab = a$, 

\item $sab\darrow$ and $sabc \simeq ac(bc)$.

\end{itemize}
\end{theorem}

In the following, we will always assume that our pcas are 
{\em nontrivial}, that is, have more than one element. 
This automatically implies that they are infinite and have $k\neq s$.

The prime example of a pca is Kleene's first model $\K_1$ that was 
already mentioned in the introduction. 
This is a model defined on the natural numbers, with application 
$$
n \cdot m = \vph_n(m).
$$
Thus $\K_1$ models the setting of classical computability theory.
We can also relativize this to an arbitrary oracle $X$, thus obtaining 
the relativized pca $\K_1^X$. 

Kleene's second model $\K_2$, 
from the book Kleene and Vesley~\cite{KleeneVesley}, 
is a pca defined on Baire space $\omega^\omega$.
Application $\alpha \cdot \beta$ in this model can be informally 
described as applying the continuous functional with code $\alpha$ 
to the real $\beta$. The original coding of $\K_2$ is a bit 
cumbersome, but it is essentially equivalent to 
$$
\alpha\cdot\beta = \Phi_{\alpha(0)}^{\alpha\oplus\beta}, 
$$
where $\Phi_e$ is the $e$-th Turing functional, and the application 
is understood to be defined if the RHS is total. 
This coding, used in \cite{ShaferTerwijn}, is easier to work with. 
See the appendix of \cite{GolovTerwijn} for a proof (and precise 
statement) of the equivalence with the original coding.

An interesting variant of $\K_2$, called the van Oosten model, is obtained 
by extending the domain to include partial functions, 
cf.\ \cite{vanOostenSeqComp}. 
$\K_2$ is uncountable, but restricting attention to computable sequences 
gives a countable pca $\K_2^\mathrm{eff}$. 
Similarly, restricting to $X$-computable 
sequences gives a pca $\K_2^X$ for every~$X$. 
In \cite{GolovTerwijn} the relations between these and other pcas 
are studied using embeddings. 

Many other examples of pcas can be found in the literature. 
For example, pcas have been extensively used in constructive 
mathematics, see 
Beeson~\cite{Beeson},  
Troelstra and van Dalen~\cite{TroelstravanDalenII}.
In particular, they have been used as a basis for models 
of constructive set theory, as in 
McCarty~\cite{McCarty}, 
Rathjen~\cite{Rathjen}, 
and Frittaion and Rathjen~\cite{FrittaionRathjen}.

A pca that has been particularly important in connection with 
combinatory algebra and lambda calculus is 
Scott's graph model~\cite{Scott}. 
This pca is a model of the lambda calculus 
(see Barendregt~\cite{Barendregt}), and it is also closely 
related to the enumeration degrees in computability 
theory, cf.\ Odifreddi~\cite{OdifreddiII}.
We will discuss this model in section~\ref{sec:cmpl}, 
where we also explain how the restriction of this model 
to the c.e.\ sets can be seen as a completion of~$\K_1$.

\section{Effective presentations of pcas}

Below, we will call a combinatory algebra $\A$ $Y$-computable if 
$\A$ has a representation such that the application $\cdot$ 
in $\A$ is $Y$-computable. 
We will also require that equality on $\A$ is $Y$-decidable.
The following definition (similar to notions 
used in \cite{GolovTerwijn}) makes this precise, 
using numberings to represent~$\A$.
Recall that a numbering is a surjective function $\gamma:\omega\to\A$.  
We think of $n\in\omega$ as a code for $\gamma(n)\in\A$.

\begin{definition}\label{def:computable}
Let $\A$ be a pca and $Y\subseteq\omega$. 
We call $\A$ {\em partial $Y$-computable\/} if there exist a numbering 
$\gamma:\omega \to \A$ and a partial $Y$-computable function $\psi$ 
such that for all $n$ and $m$, 
$\gamma(n)\cdot\gamma(m)\darrow$ in $\A$ if and only if 
$\psi(n,m)\darrow$, and 
\begin{equation}\label{pc}
\gamma(n)\cdot\gamma(m) = \gamma(\psi(n,m)).
\end{equation}
We also require that equality on $\A$ is $Y$-decidable, meaning that 
the set $\{(n,m)\mid \gamma(n)=\gamma(m)\}$ is $Y$-computable.
If $\A$ is total, i.e.\ a ca, then $\psi$ is total, and we 
simply call $\A$ {\em $Y$-computable\/}. (This is consistent with 
Definition~\ref{def:ce} below.)
\end{definition}

Notice that the numbering $\gamma$ in Definition~\ref{def:computable} 
is not required to be computable in any way.  
Also, nontrivial combinatorial algebras are never computable, 
cf.\ Barendregt~\cite[5.1.15]{Barendregt}.

Definition~\ref{def:computable} focusses on representing application 
in a pca as a p.c.\ function. 
For the record, we also mention another way to define effective  
representations. 

\begin{definition}\label{def:ce} 
We call a pca $\A$ {\em $Y$-c.e.}\ if there exist a numbering 
$\gamma:\omega \to \A$ such that the set 
\begin{equation}\label{ce}
\bigset{(n,m,k)\mid \gamma(n)\cdot\gamma(m)\darrow = \gamma(k)}
\end{equation}
is $Y$-c.e.
Again we also require that equality on $\A$ is $Y$-decidable, meaning 
that the set $\{(n,m)\mid \gamma(n)=\gamma(m)\}$ is $Y$-computable.
We call $\A$ {\em $Y$-computable\/} if the set \eqref{ce} 
is $Y$-computable.\footnote{
Note that this definition of computable pca is different from 
the definition of {\em decidable\/} pca in 
van Oosten~\cite[Definition 1.3.7]{vanOosten}, 
which refers to the decidability of equality inside the pca, 
using an element of the pca.}
\end{definition}

As an example, note that $\K_1$ is p.c.\ in the sense that $a\cdot b$ 
is a p.c.\ function on $\omega$, and that $\K_1$ is c.e.\ in the sense 
that $a\cdot b\darrow = c$ is a c.e.\ relation. 
We note here that the two definitions are equivalent:

\begin{proposition} 
A pca is partial $Y$-computable if and only if it is $Y$-c.e.
\end{proposition}
\begin{proof}
We always have c.e.\ implies p.c.: 
Given $n,m$, search for $k$ such that $(n,m,k)$ is in the set \eqref{ce},
and define $\psi(n,m)$ to be the least $k$ found. 
(Note that $\gamma$ need not be injective, so there may be 
multiple such $k$.)
Then \eqref{pc} holds for $\psi(n,m)$. 

The converse direction uses the condition that equality on $\A$ is 
decidable. Given $\A$ p.c.\ and $\psi$ satisfying \eqref{pc}, 
enumerate $(n,m,k)$ if $\psi(n,m)\darrow$ and 
$\gamma(k) = \gamma(\psi(n,m))$.
This gives an enumeration of~\eqref{ce}.
\end{proof}

Note that the above definitions are in the spirit of the 
c.e.\ structures in Selivanov~\cite{Selivanov}, where they are 
called {\em positive\/} structures. These are defined as structures 
in which the predicates are c.e., and the functions are computable. 
The latter makes sense for total functions, but in the case of pcas
we are dealing with a partial application operator, in which case 
it is natural to have this as a c.e.\ function. 

For pcas on $\omega$ (i.e.\ with $\gamma:\omega \to \A$
the identity) we have that equality on $\A$ is decidable, 
so the two notions of pca are equivalent. 
This is the type of pca that was used in 
Fokina and Terwijn~\cite{FokinaTerwijn}.
Without the condition that equality on $\A$ is decidable (or c.e.), 
it is not clear that the two definitions are equivalent in general, 
though we do not know of an example of a pca that is 
p.c.\ but not c.e.

Finally, in the case of a $Y$-computable pca or ca $\A$
(which is what we will mostly use below), the requirement that 
equality on $\A$ is $Y$-decidable actually {\em follows\/} 
from the fact that \eqref{ce} is $Y$-computable, namely 
let $n$ be such that $\gamma(n)$ is the identity 
(which exists in any pca).

\section{Embeddings and isomorphisms}

There are at least three notions of embedding for pcas, 
depending on what structure is required to be preserved. 
For example, the choice of the combinators $k$ and $s$ from 
Theorem~\ref{Feferman} can be regarded as part of the structure 
or not. For instance, Zoethout~\cite[p33]{Zoethout} does not consider 
$k$ and $s$ to be part of the structure of a pca. 
We have the following notions of embedding of pcas:

\begin{itemize}

\item Only preserve applications. 
This notion was studied in 
Bethke \cite{Bethke}, 
Asperti and Ciabattoni~\cite{AspertiCiabattoni}, 
Shafer and Terwijn~\cite{ShaferTerwijn},  
and Golov and Terwijn~\cite{GolovTerwijn}.

\item Besides applications, also preserve $k$ and $s$, for a 
particular choice of these combinators. 
This stronger notion was studied in 
Bethke, Klop, and de Vrijer~\cite{BethkeKlopVrijer1996, BethkeKlopVrijer1999}.

\item There is an even weaker notion of embedding, using the notion of 
applicative morphism, that was introduced in Longley~\cite{Longley}, 
see also Longley and Normann~\cite{LongleyNormann}. 
Applicative morphisms do not have to preserve applications; 
instead, there have to be terms in the codomain that simulate 
applications in the domain. This notion is useful in realizability 
theory, see van Oosten~\cite{vanOosten}.

\end{itemize}

Our primary interest here is the notion of embedding where $k$ and $s$ 
are not considered part of the signature, but we will also be using 
the stronger notion of embedding, especially when we talk about 
completions.
To distinguish the two, we will refer to them as weak and strong 
embeddings. (In \cite{GolovTerwijn} weak embeddings were simply 
called embeddings.) 
To distinguish applications in different pcas, we also write 
$\A\models a\cdot b\darrow$ if this application is defined in $\A$.

\begin{definition} \label{def:embedding}
For given pcas $\A$ and $\B$, an injection 
$f: \A \to \B$ is a \emph{weak embedding} if for all $a, b \in \A$, 
\begin{equation}\label{emb}
\A\models ab\darrow \; \Longrightarrow \;
\B\models f(a)f(b)\darrow \, = f(ab).
\end{equation}
If $\A$ embeds into $\B$ in this way we write $\A\hookrightarrow \B$.
If in addition to \eqref{emb}, for a specific choice of combinators 
$k$ and $s$ of $\A$, $f(k)$ and $f(s)$ serve as combinators for $\B$, 
we call $f$ a {\em strong embedding}.

A (total) combinatory algebra $\B$ is called a {\em weak completion\/} of $\A$ 
if there exists an embedding $\A \hookrightarrow \B$.
If the embedding is strong, we call $\B$ a {\em strong completion}. 

Two pcas $\A$ and $\B$ are 
{\em isomorphic\/}, denoted by $\A\cong\B$, if there exists a bijection 
$f:\A\rightarrow \B$ such that for all $a,b\in\A$, 
$ab\darrow$ if and only if $f(a)f(b)\darrow$,
and in this case 
$$
f(a)\cdot f(b) = f(ab).
$$
\end{definition}

Besides the term completion, in the literature also the term 
extension is used.  
Bethke et al.\ \cite[Definition 1.5]{BethkeKlopVrijer1999} call 
a pca $\B$ an {\em extension\/} of a pca $\A$ if 
$\A \subseteq \B$, the application $\cdot_\A$ in $\A$ is the restriction 
of application $\cdot_\B$ in $\B$ to the domain of $\cdot_\A$, 
and $\B$ and $\A$ both have the {\em same\/} combinators 
$k$ and $s$ as in Theorem~\ref{Feferman}.

Now suppose that $f:\A\hookrightarrow\B$ is a strong embedding. 
Then $f(\A)\subseteq\B$ is an extension in the above sense, where 
both $f(\A)$ and $\B$ have combinators $f(k)$ and $f(s)$. 
Note that $\A \cong f(\A)$ if we define application in $f(\A)$ by 
$f(\A)\models f(a)\cdot f(b) \darrow =f(c)$ if and only if 
$\A\models a \cdot b\darrow = c$.
So we see that total extensions and completions amount to the same thing, 
provided that in both cases we have to specify whether to also fix 
$s$ and $k$ or not.

In \cite{Terwijn2023} it is shown that 
weak and strong embeddability and completions are different:
There exists a pca that is weakly completable, but not 
strongly completable.\footnote{The argument runs as follows:
First, $\K_2$ has strong completions. Second, the counterexample 
from Bethke et al.\ \cite{BethkeKlopVrijer1999} of a pca without 
strong completions can be weakly embedded into $\K_2$. 
Hence this weak embedding cannot be made strong.}

\section{Complexity of completions of $\K_1$} \label{sec:cmpl}

It was an important open question in the 1970s whether every pca 
has a strong completion. 
The question was raised by Barendregt, Mitschke, and Scott, and discussed 
at a meeting in Swansea in 1974, cf.\ \cite{BethkeKlopVrijer1999}.
(Note that this predates Feferman's paper \cite{Feferman}.)
A negative answer was obtained by Klop~\cite{Klop}, 
see also Bethke et al.\ \cite{BethkeKlopVrijer1999}. 
Other examples of incompletable pcas can be found in 
Bethke~\cite{BethkeJSL1987} and Bethke and Klop~\cite{BethkeKlop1996}.

In contrast to these examples, $\K_1$ does have strong completions. 
This follows from the criterion given in 
Bethke et {al.}~\cite{BethkeKlopVrijer1996} 
about the existence of unique head-normal forms, 
which is satisfied in~$\K_1$.
The completion of $\K_1$ resulting from this is a certain term model 
$T(\omega)/\!\!\sim$. 
On the face of it, the equivalence relation $\sim$ is not computable, 
since it is essentially equivalence of terms in~$\K_1$. 
That indeed it cannot be computable follows from Theorem~\ref{thm:noncomp}, 
and also from the fact that computable combinatorial algebras do not exist. 

We now discuss how another famous pca can be seen as a completion 
of $\K_1$.
Scott's graph model $\G$ is a pca defined on the power set 
$\PP(\omega)$, with application defined by 
$$
X \cdot Y = \bigset{x \mid \ex u (\la x,u\ra \in X \wedge 
D_u\subseteq Y)}.
$$
Here $D_u$ as always denotes the finite set with canonical code~$u$, 
and $\la \cdot, \cdot\ra$ denotes an effective pairing function. 
$\E$ is defined as the restriction of $\G$ to the c.e.\ sets. 
That $\G$ and $\E$ are (total) combinatory algebras is implicit in 
Scott~\cite{Scott}. 
Note the close connection with enumeration reducibility 
(cf.\ Odifreddi~\cite[XIV]{OdifreddiII}): For all sets $Y$ and $Z$, 
$Z\leq_e Y$ is equivalent with $X\cdot Y = Z$ for some c.e.\ set $X$. 

In Golov and Terwijn~\cite[Corollary 7.5]{GolovTerwijn}
it was shown that $\K_1 \hookrightarrow \E$, so that we 
can see $\E$ as a weak completion of $\K_1$. 
Note that equality on $\E$ is equality of c.e.\ sets, 
which is $\Pi^0_2$-complete when we represent c.e.\ sets by their 
indices.\footnote{In the discussion of $\E$ as a model of the 
$\lambda$-calculus, Odifreddi~\cite[p858]{OdifreddiII} also defines an 
application on $\omega$ by 
$e\cdot x = \text{ index of } W_e \cdot W_x$. 
Odifreddi says that this choice of application is ``equivalent'' to $\E$. 
However, this application on $\omega$ does not give a pca, as equality 
on $\omega$ is decidable, so this would contradict that $K_1$ does not 
have a computable weak completion (cf.\ Theorem~\ref{thm:noncomp}). 
So to obtain a pca, we have to divide out by equivalence of c.e.-indices, 
which gives precisely~$\E$. Also note that the model of the 
$\lambda$-calculus really uses c.e.\ sets, not indices.}
So this is more complicated than equality in the term 
model $T(\omega)/\!\!\sim$.

We can see $\E$ as a combination of $\K_1$ and $\K_2$. 
Indeed we have 
$$
\K_1 \hookrightarrow \E \hookrightarrow \K_2
$$ 
(the latter by \cite[Corollary 6.2]{GolovTerwijn}), 
so that we can view $\E$ as a kind of middle ground between 
Kleene's models, combining the totality of $\K_2$ with the 
countability of $\K_1$. This combination famously gives a model 
of the $\lambda$-calculus, as shown in Scott~\cite{Scott}, 
see Odifreddi~\cite[XIV.4]{OdifreddiII}.

Below, we use that for an embedding $f:\K_1\hookrightarrow \A$ of 
$\K_1$ into a pca $\A$, it suffices to know the value of $f$ on 
finitely many elements. This observation was also used in 
Golov and Terwijn~\cite[Theorem 4.1]{GolovTerwijn}, and it can be 
used to bypass the fact that embeddings such as $f$ do not have 
to be computable. 
Below we give a somewhat simpler version of this trick, using 
the following lemma. 

\begin{lemma}\label{lemma}
There exist elements $t_n\in\K_1$, $n\geq 1$, such that 
for all $n$ and $m$, 
\begin{align*}
t_n \cdot m = 
\begin{cases}
n & \text{\rm if $m=0$} \\
t_{n+1} & \text{\rm if $m>0$.}
\end{cases}
\end{align*}
\end{lemma}
\begin{proof}
By the recursion theorem, let $d\in\omega$ be a code such that 
\begin{align*}
\vph_d(n,m) = 
\begin{cases}
n & \text{if $m=0$} \\
S^1_1(d,n+1) & \text{if $m>0$.}
\end{cases}
\end{align*}
Here $S^1_1$ is the primitive recursive function from the S-m-n-theorem. 
Define $t_n = S^1_1(d,n)$. W.l.o.g.\ we may assume $t_n>0$ for all~$n$. 
Then 
$\vph_{t_n}(m) = 
\vph_{S^1_1(d,n)} = 
\vph_d(n,m) = t_{n+1}$ for $m>0$ and equal to $n$ otherwise.  
\end{proof}

In Golov and Terwijn~\cite[Corollary 4.2]{GolovTerwijn} it was 
proved that if $\K_1^X\hookrightarrow \A$ is a weak embedding of 
$\K_1^X$ into a pca $\A$ with $Y$-c.e.\ inequality, then $X\leq_T Y$. 
We obtain a stronger conclusion when we assume that $\A$ is total
and $Y$-computable.

\begin{theorem} \label{thm:noncomp}
Suppose $\A$ is a $Y$-computable combinatorial algebra, and 
that $f:\K_1^X\hookrightarrow \A$ is a weak embedding. 
Then $X <_T Y$. 
\end{theorem}
\begin{proof}
The successor function $S$ in $\K_1$ satisfies $S^n(0)=n$, 
but we need an element $t$ such that the $n$-fold application 
$t\cdot\ldots \cdot t\cdot 0$ equals~$n$, 
with the convention that application associates to the left, 
not to the right. 
Let $t = t_1$ be as in Lemma~\ref{lemma}. Then for the $n$-fold 
application we have 
$t\cdot\ldots \cdot t\cdot 0 = t_n \cdot 0 = n$ for every $n>0$.
Since $f$ is an embedding, we obtain from this 
\begin{equation} \label{t}
f(n) = 
f(t\cdot\ldots \cdot t\cdot 0) = 
f(t)\cdot\ldots \cdot f(t)\cdot f(0), 
\end{equation}
with the applications repeated $n$ times. 
So we see that the image of $f$ is completely determined 
by $f(t)$ and~$f(0)$.

To show that $Y\not\leq_T X$, 
let $A$, $B$ be a $X$-computably inseparable pair 
of $X$-c.e.\ sets, and let $e$ be a code such that for all~$x$, 
\begin{align*}
\vph_e^X(x) = 
\begin{cases}
0 & \text{if $x\in A$} \\
1 & \text{if $x\in B$} \\
\uparrow & \text{otherwise}.
\end{cases}
\end{align*}
Then we have in particular that 
\begin{align*}
\K_1^X \models e \cdot x \darrow = 0 &\Longrightarrow 
\A \models f(e)\cdot f(x) = f(0), \\
\K_1^X \models e \cdot x \darrow = 1 &\Longrightarrow 
\A \models f(e)\cdot f(x) = f(1).
\end{align*}
Since $\A$ is total, for every $x$ the application 
$f(e)\cdot f(x)$ is always defined in~$\A$, and by \eqref{t} it 
is equal to a term containing only $f(t)$, $f(0)$, and application. 
Because $\A$ is $Y$-computable, we can compute a code of 
$f(e)\cdot f(x)$ effectively from~$x$. 
(All we need is $e$, and codes of $f(t)$ and $f(0)$, all of which are fixed.)
Furthermore, since the definition of $Y$-computable pca entails that 
equality on $\A$ is $Y$-decidable, we can decide with $Y$ whether 
$f(e)\cdot f(x)$ is equal to $f(0)$ or $f(1)$ or not.
It follows that the set 
$C = \{x \mid \A\models f(e)\cdot f(x) = f(0)\}$ is $Y$-computable 
and separates $A$ and $B$, 
and since $A$ and $B$ have no $X$-computable separation 
$Y$ is not $X$-computable.

That $X\leq_T Y$ can be shown using a very similar argument. 
Instead of $\vph_e^X$ above, use the characteristic function 
$\vph_d^X(x)$ which is $0$ if $x\in X$ and $1$ if $x\notin X$.
Then the rest of the argument above, replacing $e$ with $d$, 
shows that $X$ is $Y$-computable.
So we have $Y\not\leq_T X$ and $X\leq_T Y$, hence $X<_T Y$.
\end{proof}

From Theorem~\ref{thm:noncomp} we see that in particular 
$\K_1$ does not have a computable weak completion, which also
follows from the fact that combinatorial algebras 
are never computable, see Barendregt~\cite[5.1.15]{Barendregt}.
We now show that this is optimal, namely that there exist 
completions of low Turing degree. 
(Recall that $Y$ is low if $Y'\leq_T \emptyset'$). 

\begin{theorem} \label{low}
There exists a strong completion $\A$ of $\K_1$ of low Turing degree, 
that is, $\A$ is $Y$-computable such that $Y$ is low. 
\end{theorem}
\begin{proof} 
The outline of the proof is as follows. 
We first define a first-order base theory $\Cmpl$ such that each model 
of $\Cmpl$ gives rise to a strong completion of $\K_1$. 
The theory $\Cmpl$ will be consistent because we already know that 
$\K_1$ has strong completions.
We then use standard recursion theory to obtain a complete and 
consistent extension of $\Cmpl$ of low degree. 
This does not immediately give a completion of $\K_1$ of low degree, 
but we use a model theoretic argument to obtain a completion of 
the desired complexity. 

The language of $\Cmpl$ is two-sorted,\footnote{For more about 
multi-sorted languages and models see Monk~\cite[p483 ff]{Monk}. 
It is well-known that languages with finitely many sorts, as in 
our case, reduce to ordinary first-order logic by using 
predicates for the various sorts, as we do here directly. 
There is no need to keep the sorts $N$ and $A$ disjoint, so we 
have what is called a {\em lax\/} setting. 
When writing axioms for a sort, instead of writing 
$\fa a (N(a) \imp \vph)$ we also simply write $\fa a\in N. \, \vph$.}
with a predicate $N(x)$ intended to range over natural numbers, 
and a predicate $A(x)$ intended to range over a pca $\A$ that 
is a completion of $\K_1$. 
Furthermore the language has a function symbol $f$ with the 
intended meaning that $f:\K_1\to\A$ is a strong embedding. 
The language for the sort $N$ is the same as the language of 
arithmetic, and for this sort we take the axioms of $\PA$. 
The language of the sort $A$ is that of pcas, with one function 
symbol $\cdot$ for application in $\A$. 
Since $\cdot$ will be required to be total we add it as a function 
symbol, rather than as a relation symbol, which would have been
more appropriate for a partial operation. 
By arithmetization, we may assume that expressions of the form 
$\vph_a(b)\darrow =c$ are directly expressible for the sort $N$ 
for all standard numbers $a,b,c\in\omega$, where we represent a 
number $n\in\omega$ by the term $S^n(0)$.

So as axioms of $\Cmpl$ we have the following:

\begin{itemize}

\item
The axioms of $\PA$ for the sort $N$ (i.e.\ all axioms 
relative to $N$).

\item Axioms expressing that $f$ is an embedding 
from $\K_1$ to $\A$:

\begin{itemize}

\item $\fa a \in N (f(a)\in A)$.

\item $\fa a,b \in N(f(a)=f(b) \imp a=b)$.

\item For all $a,b,c\in\omega$ we have an axiom 
\begin{equation}\label{f}
\K_1 \models a\cdot b\darrow =c \;\Longrightarrow \;
\A \models f(a)\cdot f(b) = f(c).
\end{equation}
Note that the LHS can be expressed for the sort $N$ using the 
language of arithmetic, using terms $S^n(0)$ to express the natural 
number~$n$, and the RHS can be expressed for the sort~$A$. 

\end{itemize}

\item To ensure that $f$ is a strong embedding, we fix standard 
combinators $s$ and $k$ in $\K_1$ satisfying 
the axioms of Theorem~\ref{Feferman}. Note that these can be 
expressed for the sort~$N$. 
Also note that $s$ and $k$ are just standard numbers, so we do not 
need to add them to the signature. 
Next, we add axioms expressing that $f(s)$ and $f(k)$ also satisfy 
the axioms of Theorem~\ref{Feferman}, but now for the sort~$A$. 
The existence of these combinators $f(s)$ and $f(k)$ automatically 
ensures that $\A$ forms a pca. 
The fact that $\A$ should be total is handled by the fact that 
application is a function symbol in the language, so no explicit 
axiom is needed for this. 

\end{itemize}
Taken together, the axioms of $\Cmpl$ express that $f$ is a 
strong embedding from $\K_1$ to $\A$. 
Every model $M$ of $\Cmpl$ gives a strong completion of $\K_1$ as follows. 
Denote by $M\restr N$ and $M\restr A$ the part of $M$ restricted to 
the sorts $N$ and~$A$.
Then $M\restr N$ is a model of $\PA$ and $\A=M\restr A$ is a pca. 
Furthermore, the restriction of $f^M$ to the standard numbers $n\in\omega$ 
is an injection of $\omega$ into $\A$, which by \eqref{f} is an 
embedding of~$\K_1$, which is strong because 
$f(s)$ and $f(k)$ satisfy the axioms of Theorem~\ref{Feferman}. 
The values of $f^M$ on possible nonstandard elements 
of $M\restr N$ are irrelevant.

Now we let $T \subseteq 2^{<\omega}$ be a computable tree such that 
the set of infinite paths $[T]$ consists of 
all complete and consistent extensions of $\Cmpl$. 
(We encode sentences by natural numbers, so that paths in $2^\omega$ 
correspond to sets of sentences.)
The set of paths $[T]$ consists of all the separations of the sets 
\begin{align*}
\{\vph &\mid \Cmpl\vdash \vph\} \\
\{\vph &\mid \Cmpl\vdash \neg\vph\}
\end{align*}
of provable and refutable formulas.
Note that these are c.e.\ sets because $\Cmpl$ is a first-order theory, 
and they are disjoint because $\Cmpl$ is consistent, since we know by 
Bethke et al.\ \cite{BethkeKlopVrijer1996} that there exists a strong
completion of~$\K_1$. In particular the tree $T$ is infinite and 
$[T]$ is nonempty. 
By the Low Basis Theorem (Jockusch and Soare~\cite{JockuschSoare})
$T$ has a path of low Turing degree, which gives us a complete and 
consistent extension $X$ of $\Cmpl$ of low degree.

Since $X$ is consistent, it has a model $M$ by the completeness theorem, 
and by the remarks above $M$ defines a strong completion of $\K_1$, 
namely $M\restr A$. However, there is no guarantee that this completion 
is $X$-computable. 
But we do not need all of $M\restr A$;  it suffices to consider the 
smaller pca $\A$ consisting of all terms built from 
$f(n)$ for standard numbers $n\in \omega$ (represented as terms $S^n(0))$, 
and application $\cdot$.
Note that $\A$ is a pca because of the presence of $f(s)$ and~$f(k)$, 
and $\A$ is total since for $u$ and $v$ of the given form, 
$u\cdot v$ is again of this form.
Not every element of $\A$ is of the form $f(n$), for example it has
terms $f(a)\cdot f(b)$ such that $\vph_a(b)\uarrow$.
The pca $\A$ is a sub-pca of $M\restr A$ in the sense of 
\cite{ShaferTerwijn}.
To finish the proof of the theorem, we note that $\A$ is $X$-computable. 
Namely, given to terms $u$ and $v$ of the form above, we can simply 
compute their application as the term $u\cdot v$. 
Equality of terms in $\A$ is $X$-decidable because the theory $X$ 
is complete and thus contains all equalities $u=v$ and $u\neq v$ of 
such terms. 
So the sub-pca $\A$ of $M\restr A$ is total and $X$-computable, 
and hence of low degree since $X$ is low. 
\end{proof}

\section{Complete extensions of $\PA$}\label{sec:PA}

Following modern terminology, we call a Turing degree a 
PA degree if it is the degree of a complete extension of 
Peano arithmetic 
(cf.\ Downey and Hirschfeldt \cite[p84]{DowneyHirschfeldt}).
We will simply call a set PA-complete if it has PA degree. 

In this section we show that every (strong or weak) completion 
of $\K_1$ computes a PA degree, and vice versa.  
Since there exists PA-complete sets of low degree 
\cite[p87]{DowneyHirschfeldt}, 
Theorem~\ref{low} follows from the statement of this equivalence; 
however, this does not make the proof of Theorem~\ref{low} superfluous, 
since the tree $T$ from its proof is used in the proof of the equivalence.

\begin{proposition}\label{prop:PA}
Every PA-complete set computes a strong completion of~$\K_1$.
\end{proposition}
\begin{proof}
By results of Scott and Solovay (cf.\ Odifreddi~\cite[V.5.36]{Odifreddi}), 
a set $Y$ is PA-complete if and only if it can compute an element of 
every nonempty $\Pi^0_1$-class. In particular, $Y$ can compute an element 
of $[T]$ for the computable tree $T$ from the proof of Theorem~\ref{low}.
By the rest of the proof of Theorem~\ref{low}, this implies that 
$Y$ computes a strong completion of $\K_1$, namely the term model 
defined at the end of the proof. 
\end{proof}

The following result strengthens Theorem~\ref{thm:noncomp}.

\begin{theorem}\label{thm:PA}
Suppose $\A$ is a $Y$-computable combinatorial algebra, and 
that $f:\K_1\hookrightarrow \A$ is a weak embedding. Then 
$Y$ is PA-complete. 
\end{theorem}
\begin{proof}
By Jockusch and Soare \cite{JockuschSoare}, a set is PA-complete
if and only if it can compute a separation of an effectively 
inseparable pair of c.e.\ sets. (See also \cite[p86]{DowneyHirschfeldt}.)
Now let $A$, $B$ be a pair of effectively inseparable c.e.\ sets, 
for example, we can take the provable and refutable sentences of PA 
\cite[p513]{Odifreddi}.
Then the proof of Theorem~\ref{thm:noncomp} shows that $Y$ computes a
separation of $A$ and $B$, and hence $Y$ is PA-complete. 
\end{proof}

Putting Proposition~\ref{prop:PA} and Theorem~\ref{thm:PA} 
together, we obtain the following characterization:

\begin{corollary} The following are equivalent for any set $A$:
\begin{enumerate}[label={\rm (\roman*)}] 

\item $A$ set computes a weak completion of $\K_1$, 

\item $A$ set computes a strong completion of $\K_1$, 

\item $A$ is PA-complete.

\end{enumerate}
\end{corollary}

In the case of PA degrees more is known, namely that they are 
closed upwards. We do not know whether the degrees of 
(weak or strong) completions of $\K_1$ are also upwards closed.

\section{Nonstandard models of $\PA$}\label{sec:nonstandard}

As mentioned in van Oosten~\cite{vanOosten}, 
for every model $M$ of Peano Arithmetic $\PA$, we have a pca 
$\K_1(M)$ on $\omega$, with application defined by 
\begin{equation*} 
a\cdot b \darrow = c \text{ if } 
M \models \vph_a(b)\darrow = c
\end{equation*}
for all (standard) $a,b,c\in\omega$.
Note that $\K_1(M)$ is just Kleene's first model $\K_1$ 
``inside $M$''. It is a pca because we can pick combinators 
$k,s \in \K_1$ as in Theorem~\ref{Feferman} such that 
$\PA$ proves that they have the required properties.  

Note that for $a,b,c\in\omega$ we have that 
$a\cdot b\darrow = c$ in $\K_1$ 
if and only if $\ex y(T(a,b,y) \wedge U(y)=c)$, where $T$ and $U$ are 
the primitive recursive predicate and function from Kleene's normal 
form theorem (cf.\ Odifreddi~\cite{Odifreddi}). 
In a nonstandard model $M$, this $y$ can be nonstandard, so that 
{\em more\/} computations converge than in reality. 
In particular, in general we have 
\begin{equation} \label{nonstandard}
\K_1 \models a\cdot b \darrow = c \; \Longrightarrow \;
\K_1(M) \models a\cdot b \darrow = c,
\end{equation}
but not conversely. For example, consider a model $M$ of 
$\PA + \neg \con(\PA)$, where $\con(\PA)$ expresses the 
consistency of $\PA$. Such models exist by G\"odel's 
second incompleteness theorem. If we consider the p.c.\ 
function $\vph_a$ that on input $b$ searches for a proof 
of an inconsistency in $\PA$, then the computation 
$a\cdot b$ will converge in $\K_1(M)$, but not in $\K_1$ 
(assuming $\PA$ is consistent). 

Note that by \eqref{nonstandard}, we have an embedding 
$\K_1 \hookrightarrow \K_1(M)$ for every model $M$ of $\PA$.
This is in fact a strong embedding, as the same combinators 
$s$ and $k$ can be used in $\K_1(M)$.

Note that since $\PA$ proves that certain p.c.\ functions 
are nontotal, any model of $\PA$ has nontotal p.c.\ functions, 
and in particular $\K_1(M)$ is never a total pca (i.e. a ca). 
So we have:

\begin{proposition}
$\K_1(M)$ is never a weak completion of $\K_1$. 
\end{proposition}

By Tennenbaum's theorem (cf.\ Boolos and Jeffrey~\cite{BoolosJeffrey}), 
there are no computable nonstandard models of $\PA$. 
More precisely, there are no nonstandard models in which $+$ 
is computable. (This is an extension of Tennenbaum's theorem 
due to Kreisel.) It follows that there are also no nonstandard 
models that are c.e., because $+$ is a total operation, and in 
a c.e.\ model it would actually be computable, contradicting 
Kreisel's result. 
So it would seem that the pcas $\K_1(M)$ for nonstandard models~$M$ 
cannot be used for the problems about c.e.\ pcas discussed in 
computable structure theory (see \cite{FokinaTerwijn}).
However, for models $M_0$ and $M_1$ of $\PA$ we do {\em not\/} 
have in general that 
$$
M_0 \not\cong M_1 \Longrightarrow \K_1(M_0) \not\cong \K_1(M_1).
$$
To see this, let $M_0=\omega$ be the standard model, and 
let $M_1$ be a nonstandard model that has the same first order 
theory as $\omega$ (which exists by the compactness theorem). 
Then $M_0 \not\cong M_1$, but 
$\K_1(M_0) = \K_1(M_1) = \K_1$.

In particular we see that it is possible that $\K_1(M)$ is c.e.\ 
(in the sense of Definition~\ref{def:ce}) 
for a nonstandard model $M$ of $\PA$. 
This prompts the following question:

\begin{question}
What are the possible c.e.\ degrees for such $\K_1(M)$?
Can $\K_1(M)$ be c.e.\ but not equal to $\K_1$?
\end{question}

In the following we note that though $\K_1(M)$ is always noncomputable, 
it can have low degree (if we do not require that it is also c.e.).

\begin{proposition} 
$\K_1(M)$ is always noncomputable. 
\end{proposition}
\begin{proof}
Consider the computably inseparable pair of sets 
\begin{align*}
A &= \{x\in\omega \mid \vph_x(x)\darrow =0\} \\
B &= \{x\in\omega \mid \vph_x(x)\darrow =1\}.
\end{align*}
Now consider the set 
$C = \{x\in\omega \mid \K_1(M) \models x\cdot x\darrow =0\}$.
$C$ is computable from $\K_1(M)$, and it follows from 
\eqref{nonstandard} that it separates $A$ and $B$, 
from which it follows immediately that $\K_1(M)$ cannot be computable.
\end{proof}

\begin{proposition} 
$\K_1(M)$ can have low Turing degree. 
\end{proposition}
\begin{proof}
We have to show that there exists a model $M$ and a low set $Y$
such that $\K_1(M)$ is $Y$-computable 
(in the sense of Definition~\ref{def:ce}), i.e.\ such that the set 
$$
Z= \{(a,b,c)\mid \K_1(M)\models a\cdot b\darrow = c\}
$$ 
is $Y$-computable.
By Jockusch and Soare, there exists a complete and consistent 
extension $X$ of $\PA$ of low Turing degree. 
Let $M$ be a model with theory $X$. 
We can identify the set $Z$ with the set of sentences 
$a\cdot b\darrow = c$ that hold in $M$.
Since $Z$ is then just a subset of $X$ consisting of sentences 
of a specific form, there is a computable set $R$ such that 
$Z = R\cap X$, and we have $(R\cap X)' \leq_T X'\leq_T \emptyset'$ 
so that $Z$ is low.
\end{proof}

\end{document}